\documentclass[11pt]{article}
\usepackage[english]{babel}
\usepackage{csquotes}

\usepackage[a4paper,margin=1.1in]{geometry}
\usepackage{enumitem}

\usepackage{amsmath,amssymb,amsthm}

\usepackage{graphicx}
\usepackage{epstopdf} % OK on arXiv if you still have .eps; otherwise you can remove it.

\usepackage[caption=false]{subfig}

\usepackage{tikz}
\usetikzlibrary{arrows.meta}
\usepackage{pgfplots}
\usepackage[numbers,sort&compress]{natbib}
\bibpunct[, ]{[}{]}{,}{n}{,}{,}

\makeatletter
\def\NAT@def@citea{\def\@citea{\NAT@separator}}
\makeatother
\theoremstyle{plain}
\newtheorem{theorem}{Theorem}[section]
\newtheorem{lemma}[theorem]{Lemma}
\newtheorem{corollary}[theorem]{Corollary}
\newtheorem{proposition}[theorem]{Proposition}

\theoremstyle{definition}
\newtheorem{definition}[theorem]{Definition}

\theoremstyle{remark}
\newtheorem{remark}[theorem]{Remark}

\usepackage{microtype} % nicer PDF typography
\usepackage[colorlinks=true,citecolor=blue,linkcolor=blue,urlcolor=blue]{hyperref}
\newcommand{\Fix}{\mathrm{Fix}}

\begin{document}%\recd{}{}%Do not alter this line.

\title{Asymptotic Stability and Equilibrium Selection in Quasi-Feller Systems with Minimal Moment Conditions}

\author{
J.-G.~Attali\thanks{Contact: jean-gabriel.attali@devinci.fr} \\
\small Léonard de Vinci Pôle Universitaire Research Center \\
\small 12 Avenue Léonard de Vinci, 92916 Paris La Défense, France
}

\date{} % leave empty for arXiv

\maketitle

\begin{abstract}
We study equilibrium selection for invariant measures of stochastic dynamical systems
with constant step size, under persistent noise and minimal moment assumptions, in a
general quasi-Feller framework.
Such dynamics arise in projection-based algorithms, learning in games, and systems with
discontinuous decision rules, where classical Feller assumptions and small-noise or
large-deviation techniques are not applicable.

Under a global Lyapunov condition, we prove that any weak limit of invariant measures
must be supported on the set of fixed points of the associated deterministic dynamics.
Beyond localization, we establish a sharp exclusion principle for unstable equilibria:
strict local maxima and saddle points of the Lyapunov function are shown to carry zero
mass in limiting invariant measures under explicit and verifiable non-degeneracy
conditions.

Our analysis identifies a local mechanism driven by Lyapunov geometry and persistent
variance, showing that equilibrium selection in constant-step dynamics is governed by
typical fluctuations rather than rare events.
These results provide a probabilistic foundation for stability and equilibrium selection
in stochastic systems with persistent noise and weak regularity.
\end{abstract}


\vspace{0.5em}

\section{Introduction}

Stochastic algorithms with constant step size arise in a wide range of contexts,
including stochastic approximation \cite{benveniste,benaim,benaim2},
learning dynamics in games \cite{benaim3,kandori,blume},
vector quantization \cite{pages,pages2,pages2ter},
and stochastic optimization under persistent noise \cite{mertikopoulos}.
Unlike decreasing-step schemes, constant-step dynamics do not converge almost surely
to a single trajectory but instead typically admit invariant probability measures.
Understanding the structure of these invariant measures is therefore essential for
describing the long-term behavior and equilibrium selection properties of such systems.

A decisive step in this direction was made by Fort and Pag\`es \cite{pages},
who identified invariant measures as the natural object for studying constant-step
stochastic approximation.
Under suitable regularity assumptions, they established existence and tightness of
invariant measures and characterized their weak limits as the step size vanishes.
However, when the limiting deterministic dynamics is trivial, as in the flat case
$f=\mathrm{Id}$, invariance alone carries little geometric information and does not
lead to equilibrium selection without exploiting second-order instability effects.

The present work shows that this situation changes as soon as the deterministic drift
becomes nontrivial.
In that case, the invariant-measure framework becomes dynamically informative:
weak limits of invariant measures are constrained by the geometry of the limiting
deterministic dynamics and exhibit sharp localization and exclusion phenomena.
More precisely, under a global Lyapunov condition and minimal second-moment assumptions
on the noise, we show that any weak limit $\nu_0$ of invariant measures is supported on
the set where the Lyapunov function cannot strictly decrease.
Beyond localization, we establish local exclusion results for unstable equilibria
through a mechanism driven by Lyapunov geometry and persistent variance. This exclusion mechanism is inherently asymptotic and does not yield any exclusion at fixed step size.
In particular, if $x^*$ is a strict local maximum or a saddle point of a Lyapunov
function, we prove that $\nu_0(\{x^*\})=0$ under explicit non-degeneracy conditions
linking curvature and noise.
This mechanism does not rely on vanishing-noise asymptotics or large deviation theory,
and is of a fundamentally different nature from classical small-noise approaches
\cite{freidlin2012,kifer}.

A central difficulty in pursuing this program is that many stochastic algorithms of
practical interest involve genuinely discontinuous update rules.
Projection-based methods and differential-inclusion limits \cite{bianchi},
best-response dynamics in games \cite{benaim3},
and vector quantization algorithms \cite{pages2bis}
typically violate the Feller property, rendering classical approaches based on
pathwise convergence or large deviation principles inapplicable.
Imposing smoothness or strong regularity assumptions would therefore exclude precisely
the models that motivate the study of equilibrium selection under persistent noise.

To address this issue, we adopt the quasi-Feller framework introduced by Attali
\cite{attali}, which provides a minimal topological setting for Markov kernels with
discontinuous transitions.
This framework allows us to establish existence and tightness of invariant measures
under weak regularity assumptions and mild moment conditions on the noise, without
relying on Harris recurrence, minorization conditions, or exponential tail estimates.
More importantly, it enables a deterministic-dynamics interpretation of weak limits of
invariant measures in the vanishing step-size regime, even in the presence of
discontinuities.

\subsection*{Contributions}

The main contributions of this paper are as follows.

\begin{itemize}
\item \textbf{Invariant measures beyond the Feller setting.}
We establish existence and tightness of invariant probability measures for
constant-step stochastic dynamics under a quasi-Feller assumption and a coercive
Lyapunov condition, extending the invariant-measure approach of Fort and Pag\`es
\cite{pages} beyond the standard Feller framework.

\item \textbf{Localization of weak limits.}
We show that any weak limit of invariant measures as the step size vanishes is
invariant with respect to the limiting deterministic dynamics and can be localized
using Lyapunov geometry.

\item \textbf{Equilibrium selection under persistent noise.}
Under explicit non-degeneracy conditions, we prove that unstable equilibria,
including strict local maxima and saddle points of a Lyapunov function, are excluded
from the support of any weak limit of invariant measures.
These results rely only on finite second-moment assumptions on the noise and do not
invoke large deviation principles or smoothness of the drift.
\end{itemize}

\subsection*{Related work}

Exclusion of unstable points for stochastic approximation with decreasing step sizes
is classical and relies on pathwise arguments \cite{pemantle,brandiere2,benaim}.
Small-noise equilibrium selection is typically analyzed through large deviation
techniques \cite{freidlin2012,kifer}, which require smoothness and exponential tail
assumptions.

The quasi-Feller framework introduced by Attali \cite{attali} ensures existence and
stability of invariant measures for Markov chains with discontinuous transitions.
However, the invariant-measure approach of Fort and Pag\`es \cite{pages}, while
characterizing weak limits of invariant measures, does not provide a mechanism for
excluding unstable equilibria, as it does not exploit Lyapunov curvature information
or second-order variance effects.
To the best of our knowledge, exclusion of strict local maxima and saddle points for
constant-step stochastic dynamics under mere second-moment assumptions and in the
presence of discontinuities has not been previously established.

\subsection*{Organization of the paper}

Section~2 introduces the quasi-Feller framework and establishes existence and
tightness of invariant measures for constant-step stochastic dynamics.
Section~3 studies weak limits of invariant measures as the step size vanishes and
derives global localization results based on Lyapunov geometry.
Sections~4 and~5 are devoted to local exclusion results for unstable equilibria,
including strict local maxima and saddle points.
Several applications, including vector quantization algorithms and learning dynamics
in games, illustrate the scope of the theory.

\section{Quasi-Feller framework and invariance of measures}
\label{sec2}

This section introduces the quasi-Feller framework and states the conditions ensuring
existence and tightness of invariant probability measures associated with the family of
Markov kernels $(P^\gamma)_{\gamma \in (0,\gamma_0]}$.
These assumptions are sufficient for the convergence results established in
Sections~3--5.
A probability measure $\nu^\gamma$ is said to be invariant if $\nu^\gamma = \nu^\gamma P^\gamma$.
The main difficulty lies in the possible lack of continuity of the deterministic update
map and the noise, which prevents the use of classical Feller theory.
The quasi-Feller framework provides a minimal alternative adapted to discontinuous dynamics.

\subsection{Notation and functional setting}~\label{subsec:notation-and-framework}

We introduce the notation and functional setting used throughout the paper.
Unless otherwise specified, all definitions and assumptions introduced in this subsection
remain in force for the remainder of the paper.

\paragraph{State space and topology.}
Let $S$ be a Polish space equipped with its Borel $\sigma$-field $\mathcal{B}(S)$.
We also consider an auxiliary topological space $W$ endowed with $\mathcal{B}(W)$.
In applications, we typically assume $S \subset \mathbb{R}^d$ and $W \subset \mathbb{R}^p$.
In this case, $|\cdot|$ denotes the Euclidean norm, while $\|\cdot\|$ denotes operator norms.
For a subset $A \subset S$, $\partial A$ denotes its boundary.
For a mapping $H : S \to W$, $D_H$ denotes its set of discontinuity points.

\paragraph{Function spaces, measures, and kernels.}
We denote by $\mathcal{B}_b(S)$ the space of bounded Borel functions and by $C_b(S)$ the
space of bounded continuous functions.
Let $\mathcal{P}(S)$ be the set of probability measures on $S$.
A Markov kernel $P$ acts on functions and measures as
\[
Pg(x) := \int_S g(y) P(x,dy), \qquad
\mu P(A) := \int_S P(x,A)\,\mu(dx).
\]

\paragraph{Weak convergence.}
A family $(\nu^\gamma)_{\gamma>0} \subset \mathcal{P}(S)$ converges weakly to $\nu$
as $\gamma \to 0$ if $\nu^\gamma(g) \to \nu(g)$ for all $g \in C_b(S)$.

\subsection{The quasi-Feller framework}

Classical existence results for invariant measures typically assume that the transition
kernel $P$ is Feller.
This assumption fails in many relevant models, including projection-based algorithms
and best-response dynamics, where the deterministic update map is only piecewise continuous.
We therefore rely on the quasi-Feller framework introduced in \cite{attali}.

\begin{definition}[Quasi-Feller kernel]\label{def:quasi-feller}
Let $(S,d)$ be a Polish space.
A transition kernel $P(x,dy)$ is said to be \emph{quasi-Feller} if there exist
a topological space $W$, a Borel mapping $H:S\to W$, and a family of probability
measures $(Q(w,dy))_{w\in W}$ such that:
\begin{itemize}
\item[(i)] for every $g\in C_b(S)$, the mapping $w\mapsto Qg(w)$ is continuous on $W$;
\item[(ii)] for every bounded measurable function $g$, one has
\[
Pg(x)=Qg(H(x)) ;
\]
\item[(iii)] for every compact set $K\subset S$ and every
$w\in\overline{H(K)}$, one has
\[
Q(w,D_H)=0,
\]
where $D_H$ denotes the set of discontinuity points of $H$.
\end{itemize}
\end{definition}

\begin{remark}
Condition (iii) ensures that the noise prevents the process from concentrating on the
discontinuity set of $H$.
In many applications, such as the Lloyd algorithm, $D_H$ consists of lower-dimensional
boundaries and the condition is satisfied as soon as the noise distribution is absolutely
continuous.
\end{remark}

\begin{theorem}[Pakes--Hajek criteria]\label{ts1}
Let $S$ be a Polish space and $(P^\gamma)_{\gamma \in (0,\gamma_0]}$ be a family of
quasi-Feller kernels on $S$.
Let $V:S \to \mathbb{R}_+$ be a continuous coercive Lyapunov function.

\medskip
\noindent\emph{(Pakes' criterion).}
Assume that
\[
P^\gamma V - V \le \lambda(\gamma)\psi + \mu(\gamma),
\]
where $\psi$ is bounded from above and coercive.
Then, for each $\gamma$, there exists at least one $P^\gamma$-invariant probability measure.
If $\sup_{\gamma} \mu(\gamma)/\lambda(\gamma) < \infty$, the family of invariant measures is tight.

\medskip
\noindent\emph{(Hajek's criterion).}
Assume that
\[
P^\gamma V \le \alpha(\gamma)V + \beta(\gamma), \qquad \alpha(\gamma) \in (0,1).
\]
If $\sup_{\gamma} \beta(\gamma)/(1-\alpha(\gamma)) < \infty$, then the family of invariant
measures is tight and satisfies
\[
\sup_{\gamma} \nu^\gamma(V) < \infty.
\]
\end{theorem}

\subsection{Quasi-Feller framework perturbed dynamical systems}\label{sec2.2}

We consider perturbed dynamical systems of the form
\[
X_{t+1}^\gamma = f(X_t^\gamma) + \gamma e(X_t^\gamma,\varepsilon_{t+1}),
\]
where $(\varepsilon_t)$ is an i.i.d.\ noise sequence.
The associated transition kernel is given by
\[
P^\gamma g(x) = \int g(f(x)+\gamma e(x,y))\,\mu(dy).
\]
Theorem~\ref{ts1} provides existence and tightness of invariant measures under suitable
Lyapunov conditions.
We distinguish below the classical Feller case from the genuinely discontinuous setting.

\subsubsection{Continuous dynamics: the Feller case}

Under standard continuity assumptions on $f$ and the noise gain $e$,
the transition kernel $P^\gamma$ is Feller.
In this case, classical Lyapunov drift conditions imply existence and tightness
of invariant measures via Pakes' or Hajek's criterion.
The proof relies on bounded increments of the Lyapunov function and is standard.

\subsubsection{Discontinuous dynamics: the quasi-Feller case}

When $f$ is only piecewise continuous, the Feller property may fail.
To apply Theorem~\ref{ts1}, we represent $P^\gamma$ in quasi-Feller form with
$H=(f,F)$ and a Feller kernel $Q^\gamma$ on an auxiliary space.
A key requirement is that the noise does not charge the discontinuity set of $H$.

\medskip
\noindent\textbf{Discontinuity-avoidance condition.}
Let $D_{f,F}$ denote the discontinuity set of $(f,F)$.
We assume that for every compact $K\subset \mathbb{R}^d$ and every
$(w^1,w^2)\in \overline{f(K)}\times \overline{F(K)}$,
\[
\int_{\mathbb{R}^p} \mathbf{1}_{D_{f,F}}
\bigl(w^1+\gamma \widehat e(w^2,y)\bigr)\,\mu(dy)=0,
\qquad \forall \gamma\in(0,\gamma_0].
\]
Under this condition, $P^\gamma$ is quasi-Feller and Theorem~\ref{ts1} yields existence
and tightness of invariant measures under the Lyapunov drift assumptions.

Under the quasi-Feller assumptions introduced above, existence and tightness
of invariant measures still hold, provided the noise satisfies a suitable
diffusivity condition preventing the process from charging the discontinuity set.

\begin{remark}[Diffusivity condition]\label{rem:diffusivity}
In most applications, the discontinuity set has zero Lebesgue measure.
The diffusivity condition is therefore satisfied whenever the noise distribution admits
a density with respect to the Lebesgue measure.
\end{remark}

\begin{remark}[Pure noise case]\label{rem:f_id}
When $f=\mathrm{Id}$, the dynamics reduce to a state-dependent random walk.
Theorem~\ref{ts1} still ensures existence and tightness of invariant measures.
However, since every point is then a fixed point, equilibrium selection results in
Sections~4--5 become degenerate.
Accordingly, from Section~4 onward, we restrict attention to dynamics with isolated fixed
points.
\end{remark}

\section{ Asymptotic of $\nu^{\gamma}$ as $\gamma\to 0$}\label{sec:existence-invariance}

In this section, we have to separate depending on whether $f$ is continuous
or not.

\subsection{Continuous deterministic dynamics}

Our objective is to show that, when $\gamma$ goes to zero, every weak
limit   of the sequence of invariant distributions
of the PDS is an invariant
distribution of the corresponding deterministic system:

$$(S_{0})\equiv\qquad x_{t+1}=f(x_t).$$

\begin{theorem}\label{ts3}
Let $(S_{\gamma})$ be a perturbed dynamical system. Assume that:
\medskip

\noindent\textbf{(H1)} The map $f:\mathbb{R}^d\to\mathbb{R}^d$ is continuous.

\medskip
\noindent\textbf{(H2)} The noise term satisfies
\[
e(x,y)=\widehat{e}(F(x),y),
\]
where $F:\mathbb{R}^d\to\mathbb{R}^d$ and, for every $x_0\in\mathbb{R}^d$,
the map $x\mapsto \widehat{e}(x,y)$ is continuous at $x_0$ for
$\mu$-almost every $y$.

\medskip
\noindent\textbf{(H3)} For every
\[
(w^1,w^2)\in \bigcup_{K\in\mathcal{K}} \overline{f(K)} \times \overline{F(K)},
\]
one has
\[
\int_{\mathbb{R}^p}
\mathbf{1}_{D_F}\bigl(w^1+\gamma \widehat{e}(w^2,y)\bigr)\,\mu(dy)=0,
\]
where $D_F$ denotes the set of discontinuity points of $F$.

\medskip
\noindent\textbf{(H4)} There exists $\gamma_0>0$ and a tight family
$(\nu^{\gamma})_{\gamma\in(0,\gamma_0]}$ of $P^{\gamma}$-invariant probability
measures.

\medskip
Then any weak limit point $\nu^0$ of $(\nu^{\gamma})_{\gamma\in(0,\gamma_0]}$
is invariant under $f$, that is,
\[
\nu^0(g)=\nu^0(g\circ f)
\qquad\text{for all bounded Borel functions } g:\mathbb{R}^d\to\mathbb{R}.
\]
\end{theorem}

\begin{proof} Let  $g:\mathbb{R}^d\to \mathbb{R}$ be a bounded
Lipschitz function and   $[g]_{1}$ be its Lipschitz coefficient.
For any $\gamma\in]0,\gamma_0]$:
$$P^{\gamma}g(x)=\displaystyle\int_{\mathbb{R}^p}g\left(f(x)+\gamma e(x,y)
\right)\mu(dy).$$

Consider $\varepsilon>0$. There exists  a compact subset
$K_{\varepsilon}$ of $\mathbb{R}^d$ such that
$\mu(K_{\varepsilon}^c)<\varepsilon.$ Then, for all $x$ in
$\mathbb{R}^d$ :
\begin{equation}\label{eq:pg-bound}
\begin{aligned}
\bigl|P^{\gamma}g(x)-g\circ f(x)\bigr|
&=
\left|
\int_{\mathbb{R}^p}
\bigl(g(f(x)+\gamma e(x,y)) - g\circ f(x)\bigr)\,\mu(dy)
\right| \\
&\le
\int_{\mathbb{R}^p}
\bigl|g(f(x)+\gamma e(x,y)) - g\circ f(x)\bigr|\,\mu(dy) \\
&=
\int_{K_{\varepsilon}}
\bigl|g(f(x)+\gamma e(x,y)) - g\circ f(x)\bigr|\,\mu(dy) \\
&\qquad
+ \int_{K_{\varepsilon}^c}
\bigl|g(f(x)+\gamma e(x,y)) - g\circ f(x)\bigr|\,\mu(dy) \\
&\le
\gamma [g]_1 \int_{K_{\varepsilon}} |e(x,y)|\,\mu(dy)
+ 2\,\mu(K_{\varepsilon}^c)\,\|g\|_{\infty} \\
&\le
\gamma [g]_1 \int_{K_{\varepsilon}} |e(x,y)|\,\mu(dy)
+ 2\varepsilon \|g\|_{\infty}.
\end{aligned}
\end{equation}

For every $P^{\gamma}$-invariant distribution  $\nu^{\gamma}$, we
have:
\begin{eqnarray}\left|\nu^{\gamma}(g)-\nu^{\gamma}(
g\circ f)\right|& = & \left|\nu^{\gamma}(P^{\gamma}g)-\nu^{\gamma}(
g\circ f)\right|\nonumber\\
& \leq & \nu^{\gamma}\left|P^{\gamma}g-g\circ f
\right|\nonumber\\
&\leq & \gamma[g]_{1}  \displaystyle\int_{K_{\varepsilon}}|e(x,y)|\mu(dy)
+2\varepsilon\|g\|_{\infty}.\nonumber\end{eqnarray}

$g$ and $g\circ f$ being continuous functions, taking the limit
when $\gamma$ goes to zero implies that every weak limit
distribution
 $\nu^0$ of
$\left(\nu^{\gamma}\right)_{\gamma\in]0,\gamma_0]}$ satisfies:
\begin{equation}\label{sd151}\left|\nu^0(g)-\nu^0(g\circ f)\right|\leq 2\varepsilon\|g\|_{\infty}.\end{equation}
But $\varepsilon$ can be chosen as small as desired, hence:
\begin{equation}\label{sd15}\nu^0(g)=\nu^0(g\circ f).
\end{equation}

 (\ref{sd15}) is verified for any bounded Lipschitz function from $\mathbb{R}^d$ to
 $\mathbb{R}$. It implies:

$$\forall g \mbox{ continuous function from }\mathbb{R}^d \mbox{ to }\mathbb{R},\ \nu^0(g)=\nu^0(g\circ f).$$
But $\nu^0$ is a probability measure, so we have:
$$\forall g \mbox{ bounded Borel function from }\mathbb{R}^d
\mbox{ to } \mathbb{R}, \ \nu^0(g)=\nu^0(g\circ f).$$

\end{proof}

\begin{remark} $\bullet$ Theorem \ref{ts3} implies:
$$\forall A \in \mathcal{B}(\mathbb{R}^d),\ \nu^0(A)=\nu^0\left(f^{-1}(A)\right).$$
Taking $A:= \overline{f(\mathbb{R}^d)}$,  we
 have
${\rm supp}\{\nu^0\}\subset \overline{f(\mathbb{R}^d)}.$\\
$\bullet$ It is not necessary to suppose that
$\displaystyle\int_{\mathbb{R}^p}|e(x,y)|\mu(dy)<+\infty.$
\end{remark}

When $f$ is not continuous, general conclusions on the support of weak limit
distributions cannot be derived without additional assumptions.
Indeed, in the absence of further regularity, the composition $g\circ f$ need
not be $\nu^0$-almost surely continuous, which prevents the use of the invariance
property.
To recover meaningful conclusions in this setting, one must impose additional
structure on the deterministic dynamics, typically through the existence of a
global Lyapunov function.
This issue is addressed in the next section.donc

\section{Localization of invariant probability measures}\label{sec4}

This section studies the weak limits of invariant measures and shows that their support is confined to the Lyapunov plateau of the deterministic dynamics.

\subsection{ Continuous deterministic drift}\ 

Without any additional assumption, we can give a first result
using the recurrence Theorem of Poincaré \cite{krengel}

\begin{corollary}\label{cs3}
Let $(S_{\gamma})$ be a perturbed dynamical system. Assume that:
\medskip

\noindent\textbf{(H1)} The map $f:\mathbb{R}^d\to\mathbb{R}^d$ is continuous.

\medskip
\noindent\textbf{(H2)} The noise term satisfies
\[
e(x,y)=\widehat{e}(F(x),y),
\]
where $F:\mathbb{R}^d\to\mathbb{R}^d$ and, for every $x_0\in\mathbb{R}^d$,
the map $x\mapsto \widehat{e}(x,y)$ is continuous at $x_0$ for
$\mu$-almost every $y$.

\medskip
\noindent\textbf{(H3)} For every
\[
(w^1,w^2)\in \bigcup_{K\in\mathcal{K}} \overline{f(K)} \times \overline{F(K)},
\]
one has
\[
\int_{\mathbb{R}^p}
\mathbf{1}_{D_F}\bigl(w^1+\gamma \widehat{e}(w^2,y)\bigr)\,\mu(dy)=0,
\]
where $D_F$ denotes the set of discontinuity points of $F$.

\medskip
\noindent\textbf{(H4)} There exist $\gamma_0>0$ and a tight family
$(\nu^{\gamma})_{\gamma\in(0,\gamma_0]}$ of $P^{\gamma}$-invariant probability
measures.

\medskip
Then any weak limit point $\nu^0$ of $(\nu^{\gamma})_{\gamma\in(0,\gamma_0]}$
satisfies
\[
\operatorname{supp}(\nu^0) \subset B(f)
:= \left\{
x\in\mathbb{R}^d \;:\;
\exists\, \phi(n)\nearrow +\infty \text{ such that } f^{\phi(n)}(x)\to x
\right\},
\]
where $B(f)$ denotes the \emph{Birkhoff kernel} of the dynamical system.
\end{corollary}
\begin{proof} This is a straightforward application of Poincaré's Recurrence Theorem. Theorem \ref{ts3} implies that  $\nu^0$ is a stationary
distribution for the corresponding deterministic system. We can refine this approach by assuming the existence of a global Lyapunov function for the dynamical system.\end{proof}

\begin{proposition}\label{ps2}
Let $(S_{\gamma})$ be a perturbed dynamical system. Assume that:
\medskip

\noindent\textbf{(H1)} The map $f:\mathbb{R}^d\to\mathbb{R}^d$ is continuous.

\medskip
\noindent\textbf{(H2)} The noise term satisfies
\[
e(x,y)=\widehat{e}(F(x),y),
\]
where $F:\mathbb{R}^d\to\mathbb{R}^d$ and, for every $x_0\in\mathbb{R}^d$,
the map $x\mapsto \widehat{e}(x,y)$ is continuous at $x_0$ for
$\mu$-almost every $y$.

\medskip
\noindent\textbf{(H3)} For every
\[
(w^1,w^2)\in
\bigcup_{K\in\mathcal{K}} \overline{f(K)} \times \overline{F(K)},
\]
one has
\[
\int_{\mathbb{R}^p}
\mathbf{1}_{D_F}\bigl(w^1+\gamma \widehat{e}(w^2,y)\bigr)\,\mu(dy)=0,
\]
where $D_F$ denotes the set of discontinuity points of $F$.

\medskip
\noindent\textbf{(H4)} There exist $\gamma_0>0$ and a tight family
$(\nu^{\gamma})_{\gamma\in(0,\gamma_0]}$ of $P^{\gamma}$-invariant probability
measures.

\medskip
\noindent\textbf{(H5)} There exists a continuous Lyapunov function
$V:\mathbb{R}^d\to\mathbb{R}_+$ such that $V\circ f \le V$.

\medskip
Then any weak limit point $\nu^0$ of $(\nu^{\gamma})_{\gamma\in(0,\gamma_0]}$
satisfies
\[
\operatorname{supp}(\nu^0)
\subset
\{ V\circ f = V \} \cap \overline{f(\mathbb{R}^d)} .
\]
\end{proposition}

\begin{proof} Consider an increasing continuous function $\varphi$ from
$\mathbb{R}_+$ to $\mathbb{R}$ such that $\varphi(0)=0$ and
$\displaystyle\lim_{x\to +\infty}\varphi(x)=1$.

Since $\varphi(V)$ is a bounded continuous function,  Theorem
\ref{ts3} allows us to assert that  every weak limit
distribution $\nu^0$ of $\left(\nu^{\gamma}\right)_{\gamma\in]0,\gamma_0]}$ satisfies:
\begin{equation}\label{sd10}\nu^0\left(\varphi(V)\right)=\nu^0\left(\varphi(V\circ f)\right).\end{equation}

But now, $\varphi$ is increasing and $V\circ f\leq V$, so
$\varphi(V\circ f)\leq\varphi(V)$. So \eqref{sd10} implies:
$${\rm supp}\{\nu^0\}\subset \left\{\varphi(V\circ f)=\varphi(V)\right\}.$$

Finally, for $\varphi$ is injective, it is straightforward that:
$$\left\{\varphi(V\circ f)=\varphi(V)\right\}=\{V\circ f=V\}.
$$

\end{proof}

\subsubsection{Localization on the critical set}

In the Feller case, this argument goes back to Fort and Pag\`es; the present setting extends it to discontinuous kernels.

\begin{theorem}[Localization on the critical set]\label{thm:critset}
Assume that $V\in C^2(\mathbb{R}^d)$ and that the Markov kernels $(P_\gamma)_{\gamma>0}$ satisfy the following first-order Lyapunov expansion: there exist a function $\psi:\mathbb{R}^d\to[0,\infty)$ and a remainder $r_\gamma:\mathbb{R}^d\to\mathbb{R}$ such that, for all $\gamma\in(0,\gamma_0]$ and all $x\in\mathbb{R}^d$,
\begin{equation}\label{eq:PVexpansion}
P_\gamma V(x)-V(x)=-\gamma\,\psi(x)+\gamma\,r_\gamma(x),
\end{equation}
where
\begin{enumerate}
\item[(i)] $\psi(x)=\|\nabla V(x)\|^2$ (or more generally $\psi\ge 0$ and $\psi(x)=0 \iff \nabla V(x)=0$);
\item[(ii)] along any sequence $\gamma_n\downarrow 0$ such that $\nu_{\gamma_n}\Rightarrow \nu_0$, we have
\begin{equation}\label{eq:remvanish}
\int |r_{\gamma_n}(x)|\,\nu_{\gamma_n}(dx)\longrightarrow 0 .
\end{equation}
\end{enumerate}
Let $\nu_\gamma$ be an invariant probability measure for $P_\gamma$ for each $\gamma$, and let $\nu_0$ be any weak limit point of $(\nu_\gamma)$ as $\gamma\to 0$. Then
\[
\int \psi(x)\,\nu_0(dx)=0,
\qquad\text{hence}\qquad
\nu_0\big(\{x:\nabla V(x)\neq 0\}\big)=0,
\]
and in particular
\[
\mathrm{supp}(\nu_0)\subset \{x\in\mathbb{R}^d:\nabla V(x)=0\}.
\]
\end{theorem}

\begin{proof}
The argument relies on the classical Lyapunov--invariance mechanism (as in the Feller setting), but the passage to the limit here is handled through the remainder condition \eqref{eq:remvanish}, which is compatible with discontinuous perturbations.

Fix $\gamma\in(0,\gamma_0]$. Since $\nu_\gamma$ is invariant for $P_\gamma$, we have
\[
0=\int (P_\gamma V(x)-V(x))\,\nu_\gamma(dx).
\]
Using the expansion \eqref{eq:PVexpansion}, this becomes
\[
0=\int\big(-\gamma\,\psi(x)+\gamma\,r_\gamma(x)\big)\,\nu_\gamma(dx),
\]
and dividing by $\gamma>0$ yields
\begin{equation}\label{eq:psiidentity}
\int \psi(x)\,\nu_\gamma(dx)=\int r_\gamma(x)\,\nu_\gamma(dx).
\end{equation}
Let $\gamma_n\downarrow 0$ be such that $\nu_{\gamma_n}\Rightarrow \nu_0$. Taking absolute values in \eqref{eq:psiidentity} and using \eqref{eq:remvanish},
\[
0\le \int \psi(x)\,\nu_{\gamma_n}(dx)
\le \int |r_{\gamma_n}(x)|\,\nu_{\gamma_n}(dx)\longrightarrow 0.
\]
Hence $\int \psi\,d\nu_{\gamma_n}\to 0$. By lower semicontinuity of $\psi\ge 0$ (in particular, if $\psi(x)=\|\nabla V(x)\|^2$ with $V\in C^2$, then $\psi$ is continuous), we obtain
\[
\int \psi(x)\,\nu_0(dx)=0.
\]
Since $\psi\ge 0$, this implies $\psi(x)=0$ for $\nu_0$-almost every $x$, i.e.\ $\nabla V(x)=0$ $\nu_0$-a.s.\ (under (i)). Therefore
\[
\nu_0\big(\{x:\nabla V(x)\neq 0\}\big)=0,
\]
and consequently $\mathrm{supp}(\nu_0)\subset\{x:\nabla V(x)=0\}$.
\end{proof}

\subsection{ Discontinuous deterministic drift}

\begin{proposition}\label{ps21}
Let $(S_{\gamma})$ be a perturbed dynamical system. Assume that:
\medskip

\noindent\textbf{(H1)} The noise term satisfies
\[
e(x,y)=\widehat{e}(F(x),y),
\]
where $F:\mathbb{R}^d\to\mathbb{R}^d$ and, for every $x_0\in\mathbb{R}^d$,
the map $x\mapsto \widehat{e}(x,y)$ is continuous at $x_0$ for
$\mu$-almost every $y$.

\medskip
\noindent\textbf{(H2)} For every
\[
(w^1,w^2)\in
\bigcup_{K\in\mathcal{K}} \overline{f(K)} \times \overline{F(K)},
\]
one has
\[
\int_{\mathbb{R}^p}
\mathbf{1}_{D_{f,F}}\bigl(w^1+\gamma \widehat{e}(w^2,y)\bigr)\,\mu(dy)=0,
\]
where $D_{f,F}$ denotes the set of discontinuity points of the pair $(f,F)$.

\medskip
\noindent\textbf{(H3)} There exist $\gamma_0>0$ and a tight family
$(\nu^{\gamma})_{\gamma\in(0,\gamma_0]}$ of $P^{\gamma}$-invariant probability
measures.

\medskip
\noindent\textbf{(H4)} There exists a continuous Lyapunov function
$V:\mathbb{R}^d\to\mathbb{R}_+$ such that $V\circ f \le V$ and
\[
\bigl(x_n\to x \ \text{and}\ V(x_n)-V\circ f(x_n)\to 0\bigr)
\;\Longrightarrow\;
V(x)-V\circ f(x)=0 .
\]

\medskip
Then any weak limit point $\nu^0$ of $(\nu^{\gamma})_{\gamma\in(0,\gamma_0]}$
satisfies
\[
\operatorname{supp}(\nu^0)
\subset
\{ V\circ f = V \} \cap \overline{f(\mathbb{R}^d)} .
\]
\end{proposition}

\begin{proof} Focusing on the proof of Theorem \ref{ts3}
 we can oberve that for every bounded Lipschitz function $g$, we have:
\begin{equation}\label{tl}\lim_{\gamma\to 0}|\nu^{\gamma}(g\circ f)-\nu^{\gamma}(g)|=0,\end{equation}
however we cannot assert that $g\circ f$ is a $\nu^0$-$a.s.$
continuous function even if  $g$ is a continuous function.

Consider an increasing Lipschitz function $\varphi$
 from  $\mathbb{R}_+$ to $\mathbb{R}$, such that
$\varphi(0)=0$ and $\displaystyle\lim_{x\to +\infty}\varphi(x)=1$.
Define the function:
$$\widehat{\varphi}(x):=\liminf_{y\to x}\left(\varphi\left(V\right)(y)-
\varphi\left(V\circ f\right)(y)\right).$$

It is clear that we have:
$$0\leq\widehat{\varphi}\leq\varphi\left(V\right)-
\varphi\left(V\circ f\right).$$

$\widehat{\varphi}$ is lower semicontinuous. Inequality \eqref{tl}
allows us to write:
\begin{equation}\label{tl1}0=\lim_{\gamma\to 0}\nu^{\gamma}\left(\varphi(V)-
\varphi(V\circ f) \right)\geq\liminf_{\gamma\to
0}\nu^{\gamma}\left(\widehat{\varphi} \right) \geq
\nu^{0}\left(\widehat{\varphi} \right)\geq 0,\end{equation} which
gives:
\begin{equation}\label{tl2}{\rm supp}\{\nu^0\}\subset \left\{\widehat{\varphi}=0\right\}.\end{equation}

If  $x\in \left\{\widehat{\varphi}=0\right\}$, there exists a
sequence $\left(x_n\right)_{n\in\mathbb{N}}$ that converges to $x$
and verifies:
$$\lim_{n\to +\infty}\left(\varphi\left(V\right)(x_n)-
\varphi\left(V\circ f\right)(x_n)\right)=0.$$

But $\varphi$ is a injective continuous function, so we have:
$$\lim_{n\to +\infty}\left(V(x_n)-V\circ f(x_n)\right)=0,$$
which implies that  $V(x)=V\circ f(x)$ which exactly means that $x\in
\left\{V=V\circ f\right\}$. \end{proof}

\begin{remark} As a concluding note for this section, it is important to emphasize that unlike standard Feller-based approaches, our results remain valid even when the drift $f$ is discontinuous. The Quasi-Feller framework ensures that the repulsive nature of unstable points (further analyzed in Section \ref{sec4}) is not countered by the ``shocks'' of the discontinuity set $D_{f,F}$. This is a crucial feature for applying our theory to systems like the Lloyd algorithm.
\end{remark}

\subsection{Examples and applications}\label{lemn}

Assuming the existence of a global Lyapunov function $V$, any limit distribution $\nu^0$ satisfies
\[
{\rm supp}\{\nu^0\}\subset\{V\circ f=V\}\cap \overline{f(\mathbb{R}^d)}.
\]
We illustrate below how this localization principle applies to (i) genuinely discontinuous dynamics
(Lloyd's algorithm) and (ii) a smooth example exhibiting saddle-type equilibria (the Lemniscate system),
which will play a central role in the saddle-point exclusion results of Section~5.
Additional examples and technical derivations are deferred to Appendix~\ref{app:sec33}.

\subsubsection{Multi-dimensional Lloyd's algorithm (discontinuous deterministic)}

Let $C\subset \mathbb{R}^d$ be a nonempty compact convex set, and let $\mu$ be a diffuse probability
measure on $C$ with finite second moment. For $x=(x^1,\dots,x^n)\in(\mathbb{R}^d)^n$, denote by
$\bigl(C_i(x)\bigr)_{1\le i\le n}$ the associated Vorono\"{\i} cells (see Appendix~\ref{app:lloyd-md}
for the precise definition and conventions in the presence of ties).
Define the quantizer
\[
Q_x(u):=\frac{\int_{C_i(x)} v\,\mu(dv)}{\mu(C_i(x))}\quad \text{if }u\in C_i(x),
\]
and the distortion (Lyapunov) function
\[
V(x):=\mathbb{E}\bigl|Q_x(X)-X\bigr|^2.
\]
The Lloyd map is given by $f(x)=(v^1,\dots,v^n)$ where
\[
v^i:=\frac{\int_{C_i(x)} u\,\mu(du)}{\mu(C_i(x))},\qquad 1\le i\le n,
\]
whenever $\mu(C_i(x))>0$ for all $i$ (a standard assumption ensured in practice by diffuseness and
appropriate restrictions on $x$).

A classical computation shows that $V\circ f\le V$ and that
\[
V\circ f(x)=V(x)\ \Longrightarrow\ f(x)=x,
\]
so that $\{V\circ f=V\}=\Fix(f)$ coincides with the set of centroidal Vorono\"{\i}
tessellations (CVTs). Since $f$ is discontinuous across Vorono\"{\i} boundaries, the Feller property
fails in general; however, under the quasi-Feller/diffusivity assumptions of Section~\ref{sec2},
the invariant measures $(\nu^\gamma)$ exist and are tight, and any weak limit $\nu^0$ satisfies
\[
{\rm supp}\{\nu^0\}\subset \Fix(f).
\]
In dimensions $d\ge 2$, the landscape of $V$ typically features many unstable fixed points (saddles),
and the exclusion of such configurations is addressed by the results of Sections~4--5.
For the one-dimensional Lloyd algorithm and detailed derivations, see Appendix~\ref{app:lloyd-1d}.

\subsubsection{The Lemniscate syste: a smooth deterministic example }

We consider a smooth dynamics constructed from a Lyapunov function exhibiting saddle-type critical
points. Let $L:\mathbb{R}^2\to\mathbb{R}$ be defined by
\[
L(x_1,x_2):=\frac{(x_1^2+x_2^2)^2}{16}-x_1x_2,
\]
and set
\[
V(x):=\frac{L(x)^2}{2(1+L(x)^2)^{3/4}}.
\]
Following \cite{pages}, one can construct a continuous vector field $h$ of the form
$h=\nabla V+\theta$ (with an explicit smooth correction $\theta$) such that
\[
\nabla V(x)\cdot h(x)\ge 0\qquad\text{and}\qquad \{\nabla V\cdot h=0\}
=\{(0,0),(\sqrt2,\sqrt2),(-\sqrt2,-\sqrt2)\}.
\]
Consider the ODE $\dot x=-h(x)$ and fix $u_0>0$. Let $\Phi(\cdot,u_0)$ be the time-$u_0$ map of its
flow, and define the perturbed dynamical system
\[
X_{t+1}^\gamma=\Phi(X_t^\gamma,u_0)+\gamma e(X_t^\gamma,\varepsilon_{t+1}).
\]
A standard bounded-increments argument implies that $V\circ\Phi(\cdot,u_0)\le V$.
Moreover, one obtains
\[
\{V\circ \Phi(\cdot,u_0)=V\}=\{\Phi(\cdot,u_0)=\mathrm{Id}\}
=\{(0,0),(\sqrt2,\sqrt2),(-\sqrt2,-\sqrt2)\}.
\]
Therefore, any weak limit $\nu_{sdp}^{u_0,0}$ of invariant measures for the above PDS satisfies
\[
{\rm supp}\{\nu_{sdp}^{u_0,0}\}\subset \{(0,0),(\sqrt2,\sqrt2),(-\sqrt2,-\sqrt2)\}.
\]
At this stage, the localization result does \emph{not} exclude the saddle-type equilibria
$(\pm\sqrt2,\pm\sqrt2)$.
The mechanism by which persistent noise eliminates such saddles is precisely the content of the
saddle-point exclusion results proved in Section~5, where this example is revisited in detail.

The examples of Section~\ref{lemn} show that support localization alone does not rule out
unstable fixed points such as saddle-type equilibria.
In the next sections, we refine this analysis by exploiting second-order effects induced by
persistent noise, and derive explicit conditions under which invariant measures exclude
local maxima and saddle points.

\section{Exclusion of strict local maxima}
\label{sec5}

Throughout this section, we assume that $V\in C^2(\mathbb{R}^d)$ and that for
every $K>0$,
\[
\sup_{\{V\le K\}} \bigl( |\nabla V| + \|\nabla^2 V\| \bigr) < \infty .
\]

The localization results of Section~\ref{sec:existence-invariance} (and their Lyapunov refinement in
Section~\ref{sec4}) do not, by themselves, exclude unstable fixed points.
This section shows that strict local maxima of the Lyapunov function cannot be charged
by limiting invariant measures.
The mechanism is local and relies on the second-order effects induced by persistent noise.
Even under mere second-moment assumptions, these effects generate a systematic repulsion
near strict local maxima.

We first work in the continuous (Feller) setting, where the second-order expansion
of invariant measures applies.

\subsection{Global exclusion in the Feller setting}

\begin{theorem}[Exclusion of a strict local maximum]\label{ts5}
Let $(\nu_\gamma)_{\gamma>0}$ be a family of $P^\gamma$-invariant probability measures
associated with the perturbed dynamical system $(S_\gamma)$, and assume that
$f:\mathbb{R}^d\to\mathbb{R}^d$ is continuous.
Let $x^*\in\mathbb{R}^d$ satisfy $f(x^*)=x^*$ and assume that $x^*$ is a strict local maximum of the
Lyapunov function $V$, that is,
\[
\nabla V(x^*)=0
\qquad\text{and}\qquad
\nabla^2 V(x^*)\prec 0 .
\]

Assume that the following properties hold.

\medskip
\noindent\textbf{(H-Loc)} (\emph{Localization on fixed points}).  
Along any sequence $\gamma_n\to0$ such that $\nu_{\gamma_n}\Rightarrow \nu_0$, one has
\[
\nu_0\bigl(\Fix(f)\bigr)=1 .
\]

\medskip
\noindent\textbf{(H-Exp)} (\emph{Second-order expansion}).  
For every test function $W\in C_b^2(\mathbb{R}^d)$,
\[
\frac{1}{\gamma_n^2}\,\nu_{\gamma_n}(W-W\circ f)
\;\longrightarrow\;
\frac12\sum_{x\in\Fix(f)}\nu_0(\{x\})
\int e(x,y)^{\top}\nabla^2 W(x)\,e(x,y)\,\mu(dy).
\]

\medskip
\noindent\textbf{(H-Noise)} (\emph{Non-degeneracy at $x^*$}).  
For every symmetric matrix $M\prec 0$,
\[
\int e(x^*,y)^{\top} M\, e(x^*,y)\,\mu(dy) < 0 .
\]

\medskip
Finally, assume that $V$ is regular enough to construct a bounded test function
$W\in C_b^2(\mathbb{R}^d)$ satisfying
\[
W\circ f\le W, \qquad
\nabla^2 W(x^*)\prec 0, \qquad
\nabla^2 W(x)=0 \quad \text{for all } x\in\Fix(f)\setminus\{x^*\},
\]
the construction of such a function being given in the proof.

\medskip
Then any weak accumulation point $\nu_0$ of $(\nu_\gamma)$ satisfies
\[
\nu_0(\{x^*\})=0 .
\]
\end{theorem}

\begin{proof}
Let $\gamma_n\to0$ and assume that $\nu_{\gamma_n}\Rightarrow \nu_0$.
By \textbf{(H-Loc)}, the measure $\nu_0$ is supported on $\Fix(f)$.

Let $W\in C_b^2(\mathbb{R}^d)$ be a bounded test function such that
\begin{itemize}
\item $W\circ f\le W$ on $\mathbb{R}^d$,
\item $\nabla^2W(x^*)\prec 0$,
\item $\nabla^2W(x)=0$ for all $x\in\Fix(f)\setminus\{x^*\}$,
\end{itemize}
(the construction of $W$ is given below).
Applying \textbf{(H-Exp)} to $W$ yields
\[
\frac{1}{\gamma_n^2}\,\nu_{\gamma_n}(W-W\circ f)
\longrightarrow
\frac12\sum_{x\in\Fix(f)}\nu_0(\{x\})
\int e(x,y)^{\top}\nabla^2W(x)\,e(x,y)\,\mu(dy).
\]
By the choice of $W$, all terms in the sum vanish except the one corresponding to $x^*$, hence
\[
\frac{1}{\gamma_n^2}\,\nu_{\gamma_n}(W-W\circ f)
\longrightarrow
\frac12\,\nu_0(\{x^*\})
\int e(x^*,y)^{\top}\nabla^2W(x^*)\,e(x^*,y)\,\mu(dy).
\]

Since $W\circ f\le W$, we have $W-W\circ f\ge 0$ and therefore
$\nu_{\gamma_n}(W-W\circ f)\ge 0$ for every $n$, so the left-hand side above is nonnegative.
On the other hand, $\nabla^2W(x^*)\prec 0$ and \textbf{(H-Noise)} imply
\[
\int e(x^*,y)^{\top}\nabla^2W(x^*)\,e(x^*,y)\,\mu(dy)<0.
\]
Consequently, the only possibility is $\nu_0(\{x^*\})=0$.
\end{proof}

\begin{lemma}[Construction of a test function]\label{lem:W-construction}
Assume that $f:\mathbb{R}^d\to\mathbb{R}^d$ is continuous and that there exists a function
$V:\mathbb{R}^d\to\mathbb{R}$ satisfying the following properties:

\begin{itemize}
\item $V\in C^2$ on a forward invariant sublevel set $\{V\le K\}$;
\item $V\circ f\le V$ on $\mathbb{R}^d$;
\item $x^*$ is a strict local maximum of $V$, i.e. $\nabla V(x^*)=0$ and $\nabla^2V(x^*)\prec0$;
\item $V(x^*)>V(x)$ for all $x\in\mathrm{Fix}(f)\setminus\{x^*\}$.
\end{itemize}

Then there exists a bounded function $W\in C_b^2(\mathbb{R}^d)$ such that:
\begin{enumerate}
\item $W\circ f\le W$ on $\mathbb{R}^d$;
\item $\nabla W(x^*)=0$ and $\nabla^2W(x^*)\prec0$;
\item $\nabla^2W(x)=0$ for all $x\in\mathrm{Fix}(f)\setminus\{x^*\}$.
\end{enumerate}
\end{lemma}

\begin{proof}
\emph{Step 1: regularization of the Lyapunov function.}
Choose a nondecreasing function $\varphi_1\in C^2(\mathbb{R})$ such that
$\varphi_1(t)=t$ for $t\le K$ and such that $\varphi_1'$ and $\varphi_1''$ have compact support.
Define $V_1:=\varphi_1\circ V$.
Then $V_1\circ f\le V_1$ on $\mathbb{R}^d$, $\nabla^2V_1$ is bounded, and
$\nabla^2V_1(x^*)=\nabla^2V(x^*)\prec0$.

\medskip
\emph{Step 2: separation of fixed points.}
Since $x^*$ is a strict local maximum of $V$ and $V(x^*)>V(x)$ for all
$x\in\mathrm{Fix}(f)\setminus\{x^*\}$, there exists a $C^2$ nondecreasing function
$\varphi_2:\mathbb{R}\to\mathbb{R}$ such that $\varphi_2$ is strictly increasing in a neighborhood
of $V_1(x^*)$ and constant on a neighborhood of
$\{V_1(x):x\in\mathrm{Fix}(f)\setminus\{x^*\}\}$.
Set $V_2:=\varphi_2\circ V_1$.
Then $\nabla^2V_2(x^*)\prec0$ and $\nabla^2V_2(x)=0$ for all
$x\in\mathrm{Fix}(f)\setminus\{x^*\}$.

\medskip
\emph{Step 3: boundedness.}
Finally, choose a nondecreasing function $\varphi_3\in C^2(\mathbb{R})$ which is constant outside
a compact interval and define $W:=\varphi_3\circ V_2$.
Then $W\in C_b^2(\mathbb{R}^d)$, $W\circ f\le W$ on $\mathbb{R}^d$, and
\[
\nabla^2W(x^*)\prec0,
\qquad
\nabla^2W(x)=0 \quad \text{for all } x\in\mathrm{Fix}(f)\setminus\{x^*\}.
\]
This concludes the proof.
\end{proof}

\begin{remark}
If a $C^2$ Lyapunov function $V$ satisfies $V\circ f \le V$ and
$\lim_{|x|\to\infty} V(x)=+\infty$, then all its sublevel sets $\{V\le K\}$ are compact.
In particular, any assumption requiring compactness or forward invariance of
the sublevel sets of $V$ is automatically satisfied under these conditions.
\end{remark}

\begin{remark}
As in \cite{pages}, one may replace the negative definiteness assumption on $\nabla^2 V(x^*)$
by a weaker condition.
It is sufficient that $x^*$ be a strict local maximum and that
\[
\int_{\mathbb{R}^p}
e(x^*,y)^{\top}\,\nabla^2 \widehat W(x^*)\,e(x^*,y)\,\mu(dy) < 0,
\]
which may occur even if $\nabla^2 \widehat W(x^*)$ is not negative definite (one only needs
$\nabla^2 \widehat W(x^*)\le 0$).
For instance, it suffices to assume that there exists an eigenvector $u_1^*$ of
$\nabla^2 V(x^*)$ (and hence of $\nabla^2 \widehat W(x^*)$) such that
\[
\mathbb{P}\bigl(e(x^*,\varepsilon)\cdot u_1^*>0\bigr)>0.
\]
\end{remark}

\subsection{Application to the Lemniscate example}

Recall that for $u>0$, any weak limit $\nu_{sdp}^{u,0}$ associated with the time-$u$ map
$\Phi(\cdot,u)$ satisfies
\[
\mathrm{supp}(\nu_{sdp}^{u,0})\subset\bigl\{(0,0),(\sqrt2,\sqrt2),(-\sqrt2,-\sqrt2)\bigr\}.
\]
Moreover,
\[
\nabla^2 V(\sqrt2,\sqrt2)=\nabla^2 V(-\sqrt2,-\sqrt2)
=-\frac{5}{2^{11/4}}
\begin{pmatrix}
1 & 0\\
0 & 1
\end{pmatrix}.
\]
Therefore, Theorem~\ref{ts5} implies
\[
\nu_{sdp}^{u,0}\bigl(\{(\sqrt2,\sqrt2),(-\sqrt2,-\sqrt2)\}\bigr)=0,
\]
so that the only possible weak limit is $\delta_{(0,0)}$, and hence
\[
\nu_{sdp}^{u,\gamma}\Rightarrow \delta_{(0,0)}\qquad\text{as }\gamma\to0.
\]

\begin{figure}[ht]
    \centering
    \includegraphics[width=0.7\textwidth]{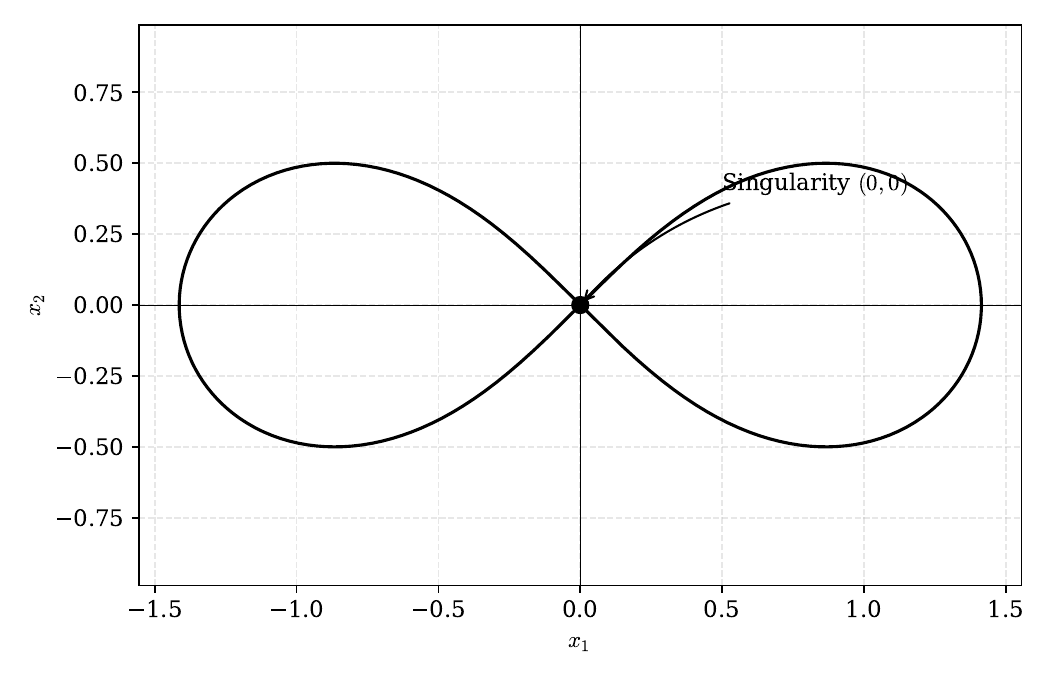}
    \caption{Lemniscate of Bernoulli dynamics. The plot illustrates the trajectory converging toward the central singularity at $(0,0)$, characterizing the asymptotic behavior of the system.}
    \label{fig:lemniscate}
\end{figure}

\subsection{Local exclusion in the Feller setting}

We next state a localized version of the exclusion result, obtained by replacing the
global Lyapunov structure by a local one.

\begin{proposition}[Local exclusion of a strict local maximum]\label{ps6}
Let $(\nu_\gamma)_{\gamma>0}$ be a family of $P^\gamma$-invariant probability
measures associated with the perturbed dynamical system $(S_\gamma)$, and assume
that $f:\mathbb{R}^d\to\mathbb{R}^d$ is continuous.
Let $x^*\in\Fix(f)$.

Assume that there exists $\varepsilon_0>0$ such that the following local framework
holds in $B(x^*,\varepsilon_0)$.

\medskip
\noindent\textbf{(L1) Local Lyapunov structure.}
There exists a function $V\in C^2\bigl(B(x^*,\varepsilon_0)\bigr)$ such that
\begin{itemize}
\item the ball $B(x^*,\varepsilon_0)$ is locally forward invariant:
\[
f^{-1}\!\bigl(B(x^*,\varepsilon_0)\bigr)\subset B(x^*,\varepsilon_0);
\]
\item $x^*$ is a strict local maximum of $V$:
\[
\nabla V(x^*)=0, \qquad \nabla^2 V(x^*)\prec 0,
\]
and
\[
V\circ f(x)<V(x)<V(x^*),
\qquad
\forall x\in B(x^*,\varepsilon_0)\setminus\{x^*\}.
\]
\end{itemize}

\medskip
\noindent\textbf{(L2) Local localization on fixed points.}
Along any sequence $\gamma_n\to0$ such that $\nu_{\gamma_n}\Rightarrow\nu_0$, one has
\[
\nu_0\bigl(B(x^*,\varepsilon_0)\cap\Fix(f)\bigr)=1 .
\]

\medskip
\noindent\textbf{(L3) Local second-order expansion.}
For every test function $W\in C_b^2(\mathbb{R}^d)$ with support contained in
$B(x^*,\varepsilon_0)$,
\[
\frac{1}{\gamma_n^2}\,\nu_{\gamma_n}(W-W\circ f)
\;\longrightarrow\;
\frac12\sum_{x\in B(x^*,\varepsilon_0)\cap\Fix(f)}\nu_0(\{x\})
\int e(x,y)^{\top}\nabla^2 W(x)\,e(x,y)\,\mu(dy).
\]

\medskip
\noindent\textbf{(L4) Local non-degeneracy of the noise.}
For every symmetric matrix $M\prec 0$,
\[
\int e(x^*,y)^{\top} M\, e(x^*,y)\,\mu(dy) < 0 .
\]

\medskip
Assume moreover that $V$ is regular enough to construct a function
$W\in C_b^2(\mathbb{R}^d)$, supported in $B(x^*,\varepsilon_0)$, such that
\[
W\circ f\le W, \qquad
\nabla^2 W(x^*)\prec 0, \qquad
\nabla^2 W(x)=0 \;\;\text{for all }
x\in B(x^*,\varepsilon_0)\cap\Fix(f)\setminus\{x^*\}.
\]

\medskip
Then any weak accumulation point $\nu_0$ of $(\nu_\gamma)$ satisfies
\[
\nu_0(\{x^*\})=0 .
\]
\end{proposition}

\begin{proof}
Define
\[
\alpha_0:=\inf_{x\in \partial B(x^*,\varepsilon_0)}\bigl(V(x^*)-V(x)\bigr)>0,
\]
and, for $\alpha\in(0,\alpha_0]$, set
\[
\Omega_{\alpha}:=\Bigl\{x\in B(x^*,\varepsilon_0):\ V(x^*)-\alpha\le V(x)\Bigr\}.
\]
By the local Lyapunov property, there exists $\alpha_1\in(0,\alpha_0]$ such that
\[
f(\Omega_{\alpha_1})\subset \Omega_{\alpha_0}.
\]

Choose a nondecreasing function $\varphi\in C^{\infty}(\mathbb{R})$ such that
\[
\varphi\bigl(V(x^*)\bigr)=1,\qquad
\varphi(y)=0 \ \text{for all } y\le V(x^*)-\alpha_1,\qquad
\varphi'\bigl(V(x^*)\bigr)=1.
\]
Define the function $V^*$ by
\[
V^*(x):=
\begin{cases}
\varphi\circ V(x), & x\in \Omega_{\alpha_1},\\[0.2em]
0, & x\notin \Omega_{\alpha_1}.
\end{cases}
\]
Then $V^*\in C_b^2(\mathbb{R}^d)$, its Hessian is bounded, and one readily checks that
$V^*\circ f\le V^*$ on $\mathbb{R}^d$.

Moreover,
\[
\nabla V^*(x^*)=0,
\qquad
\nabla^2 V^*(x^*)=\nabla^2 V(x^*),
\]
and $\nabla^2 V^*(x)=0$ for all $x\in\Omega_{\alpha_1}\setminus\{x^*\}$.
The conclusion follows by applying the same second-order expansion argument as in the proof of
Theorem~\ref{ts5}.
\end{proof}

\begin{remark}
Assumption~(L1) is restrictive, as it prevents the deterministic dynamics from
leaving and subsequently re-entering the neighborhood $B(x^*,\varepsilon_0)$.
This emphasizes an important distinction between the deterministic and the
perturbed systems: when $\gamma$ is small, the stochastic dynamics remains
essentially local near $x^*$, whereas the deterministic dynamics may exhibit
recurrent returns to this neighborhood in the absence of a Lyapunov structure.
\end{remark}

\subsection{Extension to the quasi-Feller framework}

The second-order expansion used in the continuous (Feller) setting relies on continuity properties
of auxiliary functions of the form
\[
G_{\gamma}(x)
:=\frac{1}{2}\int_{\mathbb{R}^p}
|e(x,y)|^2
\left(
\omega\!\left(\nabla^2 g,f(x),\gamma e(x,y)\right)
\wedge
\|\nabla^2 g\|_{\infty}
\right)
\,\mu(dy),
\]
where $\omega(\cdot,\cdot,\cdot)$ denotes the modulus of continuity of $\nabla^2 g$.
In the quasi-Feller setting, these functions need not be continuous, so the direct weak
convergence argument may fail.
However, the argument remains valid if $G_{\gamma}$ is $\nu^0$-almost surely continuous along
a convergent subsequence.

\begin{theorem}[Exclusion under quasi-Feller assumptions]\label{ts52}
Let $(\nu^{\gamma})_{\gamma\in(0,\gamma_0]}$ be a tight family of $P^{\gamma}$-invariant
probability measures associated with the perturbed dynamical system $(S_{\gamma})$,
and let $\nu^0$ be a weak limit point of $(\nu^{\gamma})_{\gamma\in(0,\gamma_0]}$.

Assume that the quasi-Feller framework introduced in this section holds, namely:
\begin{itemize}
\item the noise admits a quasi-Feller factorization and satisfies the
non-concentration condition;
\item the centeredness and integrability assumptions ensuring the validity of
the second-order expansion are satisfied;
\item the Lyapunov structure and finiteness of fixed points on sublevel sets
hold, and fixed points are compatible with the quasi-Feller discontinuities.
\end{itemize}

Let $x^*\in\Fix(f)$ be a strict local maximum of the Lyapunov function $V$, that is,
\[
\nabla V(x^*)=0
\qquad\text{and}\qquad
\nabla^2 V(x^*)\prec 0 .
\]
Then
\[
\nu^0(\{x^*\})=0 .
\]
\end{theorem}

\begin{proof}
By Proposition~\ref{ps21}, the support of any weak limit $\nu^0$ is contained in
\[
\operatorname{supp}(\nu^0)\subset \{V\circ f = V\} = \Fix(f).
\]
By the quasi-Feller assumptions, in particular the compatibility condition
$\Fix(f)\cap D_{f,F}=\emptyset$, the mappings $G_\gamma$ are $\nu^0$-almost surely
continuous on $\Fix(f)$. As a consequence, the second-order expansion of invariant
measures applies along any sequence $\gamma_n\to 0$ such that
$\nu_{\gamma_n}\Rightarrow \nu^0$.
The conclusion then follows exactly as in the proof of
Theorem~\ref{ts5}.
\end{proof}

\subsection{Application to noisy best-response dynamics in a coordination game}
\label{sec:game_theory}

We briefly illustrate the quasi-Feller setting on a discontinuous decision rule arising in
coordination games.
Let $x_t\in[0,1]$ denote the proportion of agents playing strategy $A$ and consider the pure
best-response map
\[
f(x)=
\begin{cases}
1, & x>1/2,\\
0, & x<1/2,\\
1/2, & x=1/2.
\end{cases}
\]
The discontinuity at $x=1/2$ is intrinsic and is often bypassed in the literature by smoothing
$f$ through a logit response; see, e.g., \cite{blume,benaim3}.
Here we keep the pure best response and consider the constant-step perturbed system
\[
X_{t+1}^{\gamma}=\Pi_{[0,1]}\bigl(f(X_t^{\gamma})+\gamma \varepsilon_{t+1}\bigr),
\]
where $\Pi_{[0,1]}$ denotes projection onto $[0,1]$ and $(\varepsilon_t)$ are i.i.d.\ centered
perturbations with finite second moment (for instance, $\varepsilon_t\sim\mathcal U([-1,1])$).

Consider the $C^2$ Lyapunov function $V(x)=x^2(1-x)^2$.
Then $V\circ f\le V$, with strict inequality on $[0,1]\setminus\{0,1/2,1\}$, and
$\{V\circ f=V\}=\{0,1/2,1\}$.
Under the quasi-Feller non-concentration condition, Proposition~\ref{ps21} localizes any weak
limit $\nu^0$ to $\{0,1/2,1\}$.
Moreover, since $x^*=1/2$ is a strict local maximum of $V$, Theorem~\ref{ts52} yields
\[
\nu^0(\{1/2\})=0.
\]
Consequently, limiting invariant measures charge only the stable conventions $\{0,1\}$.

\medskip

The results of this section show that strict local maxima of the Lyapunov function are systematically
excluded from the support of limiting invariant measures.
The argument is purely local and relies on second-order variance effects rather than on rare-event
asymptotics.
In particular, only second-moment assumptions are used, in contrast with approaches based on
large deviations or exponential integrability assumptions.
The next section addresses saddle-type equilibria, for which the Hessian admits both stable and
unstable directions and the present argument no longer applies directly.

\section{Exclusion of saddle-type fixed points}
\label{sec6}

This section extends the previous analysis to saddle points of the Lyapunov function.
In contrast with strict local maxima, a saddle point possesses both stable and unstable
directions, so that purely local trapping conditions such as
$f^{-1}(B(x^*,\varepsilon_0))\subset B(x^*,\varepsilon_0)$ are typically unreasonable.
We show that, provided the noise excites the unstable directions, such equilibria cannot
be charged by limiting invariant measures.
The proof relies on a projection of the dynamics onto the unstable subspace and combines
a second-order expansion with a divergence argument along iterates of the linearized map.

\noindent\textbf{Remark.}
Importantly, the regularity assumptions required in this section concern only the deterministic drift $f$, which is assumed to be locally $C^2$ in a neighborhood of the saddle point; the noise term $e$ may remain discontinuous, provided it satisfies the quasi-Feller non-concentration condition.

\subsection{A general exlcusion theorem}

Let $x^*$ be a fixed point of $f$ and assume that $f$ is differentiable at $x^*$.
When $x^*$ is a saddle point for $V$, the linearization typically has an unstable component,
and one expects $\rho(Df(x^*))\ge 1$.
Accordingly, the tangent space admits a $Df(x^*)$-invariant decomposition
\[
\mathbb{R}^d = E_{<1}^* \oplus E_{\geq 1}^*,
\]
where $E_{<1}^*$ (resp.\ $E_{\ge 1}^*$) is the generalized invariant subspace associated with
eigenvalues of modulus strictly less than $1$ (resp.\ greater than or equal to $1$).
We denote by $\Pi^*$ the projector onto $E_{\ge 1}^*$ along $E_{<1}^*$.

We first recall a simple linear-algebraic lemma which will be used in the proof of the saddle-point exclusion result.

\begin{lemma}\label{ls1}
Let $F$ be a finite-dimensional Euclidean space and let $A\in\mathcal L(F)$ such that
$\rho(A^{-1})\le 1$.
Then
\[
\Bigl\{v\in F:\ \lim_{p\to\infty}A^p v=0\Bigr\}=\{0\}.
\]
\end{lemma}

\begin{proof}
Let $G:=\{v\in F:\ A^p v\to0\}$.
Then $G$ is a nontrivial linear subspace only if $G\neq\{0\}$.
Assume $G\neq\{0\}$ and let $q=\dim(G)$ with a basis $(v_1,\ldots,v_q)$.
Define on $G$ the norm $|\sum_{i=1}^q\lambda_i v_i|_1:=\sum_{i=1}^q|\lambda_i|$.
For $v=\sum\lambda_i v_i\in G$,
\[
|A^p v|\le \sum_{i=1}^q|\lambda_i|\,|A^p v_i|
\le |v|_1 \sup_{1\le i\le q}|A^p v_i|.
\]
Since all norms are equivalent on $G$, there exists $\alpha>0$ such that $|v|_1\le\alpha|v|$,
and hence
\[
\|A^p|_G\|\le \alpha\sup_{1\le i\le q}|A^p v_i|.
\]
By definition of $G$, the right-hand side tends to $0$, so $\|A^p|_G\|\to0$.
This is impossible since $\rho(A^p|_G)=\rho((A|_G)^p)\ge 1$ under $\rho(A^{-1})\le1$.
Thus $G=\{0\}$.
\end{proof}

We are now in a position to state and prove the exclusion result for saddle-type fixed points.

\begin{theorem}[Exclusion of a saddle-type fixed point]\label{ts6}
Let $(\nu^\gamma)_{\gamma\in(0,\gamma_0]}$ be a tight family of $P^\gamma$-invariant probability measures
associated with the perturbed dynamical system $(S_\gamma)$, and let $\nu^0$ be a weak limit of
$(\nu^\gamma)$.
Assume that $f:\mathbb{R}^d\to\mathbb{R}^d$ is $C^2$, and that there exists a function
$V\in C^2(\mathbb{R}^d)$ such that:

\begin{itemize}
\item \textbf{($\alpha$)} $V\circ f\le V$ and $\{V\circ f=V\}=\Fix(f)$;

\item \textbf{($\beta$)} if $(x_n)_{n\ge1}$ is a sequence such that
$x_n\to x$ and $V(x_n)-V\circ f(x_n)\to0$, then $V(x)-V\circ f(x)=0$;

\item \textbf{($\delta$)} for all $K>0$, the sublevel set $\{V\le K\}$ is forward invariant under $f$ and
\[
\sup_{\{V\le K\}}
\Bigl(
|\nabla V|+\|\nabla^2 V\|+\|Df\|+\|D^2 f\|
\Bigr)<\infty;
\]

\item \textbf{(fin)} for all $v>0$, the set $\Fix(f)\cap\{V\le v\}$ is finite.
\end{itemize}

Let $x^*\in\Fix(f)$ be such that $\nabla V(x^*)=0$ and $x^*$ is a critical point of saddle type for the Lypunov function $V$
in the sense that the linearized Lyapunov dissipation at $x^*$ statisfies assumption \textbf{(a)} below.

Assume moreover that the following properties hold.

\medskip
\noindent\textbf{(H-Loc)} (\emph{Localization on fixed points}).
Along any sequence $\gamma_n\to0$ such that $\nu^{\gamma_n}\Rightarrow\nu_0$, one has
$\nu_0(\Fix(f))=1$.

\medskip
\noindent\textbf{(H-Exp)} (\emph{Second-order expansion}).
For every $W\in C_b^2(\mathbb{R}^d)$,
\[
\frac{1}{\gamma_n^2}\,\nu^{\gamma_n}(W-W\circ f)
\longrightarrow
\frac12\sum_{x\in\Fix(f)}\nu_0(\{x\})
\int_{\mathbb{R}^p} e(x,y)^{\top}\nabla^2W(x)\,e(x,y)\,\mu(dy).
\]

\medskip
\noindent\textbf{(a)} (\emph{Saddle-type nondegeneracy of the linearized Lyapunov dissipation at $x^*$}).
There exists $c>0$ such that
\[
\nabla^2V(x^*) - Df(x^*)^{\top}\,\nabla^2V(x^*)\,Df(x^*) \ \ge\ c\,I.
\]

\medskip
\noindent\textbf{(b)} (\emph{Noise excites the noncontracting directions}).
Let $E_{\ge1}^*$ be the $Df(x^*)$-invariant subspace associated with eigenvalues of modulus at least $1$,
and let $\Pi^*$ denote the projection onto $E_{\ge1}^*$ along a complementary invariant subspace.
Assume
\[
\int_{\mathbb{R}^p}\|\Pi^* e(x^*,y)\|^2\,\mu(dy)\ >\ 0 .
\]

Then
\[
\nu^0(\{x^*\})=0.
\]
\end{theorem}

\begin{remark}
The unstable subspace $E_{\ge 1}^*$ is necessarily nontrivial.
If $E_{<1}^*=\{0\}$, then $\Pi^*$ coincides with the identity on $\mathbb{R}^d$ and
assumption~(b) reduces to $\int |e(x^*,y)|\,\mu(dy)>0$.
\end{remark}

%\subsection{Proof of Theorem~\ref{ts4}}

\begin{proof}
\textsc{Step 1. Construction of iterated test functions.}
Fix $K>V(x^*)$ and let $V^\star=\varphi_1\circ V$ be the bounded $C^2$ function introduced in the proof
of Theorem~\ref{ts5}, where $\varphi_1\in C^\infty(\mathbb{R}_+,\mathbb{R}_+)$ is nondecreasing,
$\varphi_1(0)=0$, $\varphi_1'(y)>0$ for $y\in[0,K)$, and $\varphi_1$ is constant on $[K,\infty)$.
Then $V^\star\in C_b^2(\mathbb{R}^d)$, $\nabla V^\star$ vanishes on $\{V\le K\}^c$, and
$V^\star\circ f\le V^\star$.

For $p\ge1$, define $W_p:=V^\star\circ f^p$.
Since $V\circ f\le V$, the sublevel set $\{V\le K\}$ is forward invariant under $f$, hence under $f^p$.
Using the boundedness of $Df$ and $D^2f$ on $\{V\le K\}$ together with the boundedness of
$\nabla V^\star$ and $\nabla^2V^\star$ (with $\nabla^2V^\star$ vanishing on $\{V\le K\}^c$),
we obtain that $W_p\in C_b^2(\mathbb{R}^d)$ and that $\nabla^2W_p$ is bounded.

Moreover, since $f(x^*)=x^*$ and $\nabla V^\star(x^*)=0$, we have
$\nabla W_p(x^*)=0$ and
\begin{equation}\label{eq:saddle-hess}
\begin{aligned}
\nabla^2 W_p(x^*)
&= (Df(x^*)^p)^\top \nabla^2V^\star(x^*)\,Df(x^*)^p \\
&= \varphi_1'\!\bigl(V(x^*)\bigr)\,
(Df(x^*)^p)^\top \nabla^2V(x^*)\,Df(x^*)^p .
\end{aligned}
\end{equation}
If $\widehat x\in\mathrm{Fix}(f)\cap\{V\le K\}^c$, then $\nabla^2W_p(\widehat x)=0$.

Let $\gamma_n\searrow0$ be such that $\nu^{\gamma_n}\Rightarrow\nu^0$.
Since $\Omega_v:=\mathrm{Fix}(f)\cap\{V\le K\}$ is finite, the second-order expansion applied to $W_p$
yields, for each fixed $p\ge1$,
\begin{equation}\label{eq:saddle-exp}
\lim_{n\to\infty}\frac{1}{\gamma_n^2}\,
\nu^{\gamma_n}(W_p-W_{p+1})
=
\frac12\sum_{x_0\in A_v}\nu^0(\{x_0\})
\int_{\mathbb{R}^p} e(x_0,y)^\top \nabla^2W_p(x_0)\,e(x_0,y)\,\mu(dy).
\end{equation}
Since $V^\star\ge V^\star\circ f$, we have $W_p\ge W_{p+1}$, hence the left-hand side
of~\eqref{eq:saddle-exp} is nonnegative.

\medskip
\textsc{Step 2. Divergence of the unstable quadratic form.}
By $V\circ f\le V$, the matrix
\[
\nabla^2V(x^*)-Df(x^*)^\top\nabla^2V(x^*)Df(x^*)
\]
is nonnegative. By assumption~\textbf{(a)}, it is nondegenerate, hence there exists $c>0$ such that
\[
\nabla^2V(x^*)-Df(x^*)^\top\nabla^2V(x^*)Df(x^*)\ge c\,I .
\]
Equivalently, for all $y$,
\[
(Df(x^*)e(x^*,y))^\top\nabla^2V(x^*)(Df(x^*)e(x^*,y))
\le
e(x^*,y)^\top\nabla^2V(x^*)e(x^*,y)-c\,\|e(x^*,y)\|^2 .
\]
Iterating this inequality yields, for all $p\ge1$ and all $y$,
\begin{equation}\label{eq:saddle-iter}
\begin{aligned}
&(Df(x^*)^p e(x^*,y))^\top \nabla^2V(x^*)
(Df(x^*)^p e(x^*,y)) \\
&\qquad\le
e(x^*,y)^\top \nabla^2V(x^*) e(x^*,y)
- c\sum_{k=0}^{p-1}\|Df(x^*)^k e(x^*,y)\|^2 .
\end{aligned}
\end{equation}

Let $E_{\ge1}^\star$ denote the $Df(x^*)$-invariant subspace associated with eigenvalues
of modulus at least one, and let $\Pi^\star$ be the corresponding projection.
Since $\Pi^\star$ commutes with $Df(x^*)$, there exists $C>0$ such that
\[
\|Df(x^*)^k e(x^*,y)\|
\ge
C\,\|Df(x^*)^k \Pi^\star e(x^*,y)\|
\quad\text{for all }k\ge0.
\]

By Lemma~\ref{ls1}, $\Pi^\star e(x^*,y)\neq 0$ implies that
$\|Df(x^*)^k \Pi^\star e(x^*,y)\|\not\to 0$ as $k\to\infty$;
since a necessary condition for the convergence of a series is that its
general term tends to zero, it follows that
\[
\sum_{k=0}^{p-1}\|Df(x^*)^k \Pi^\star e(x^*,y)\|^2 \longrightarrow +\infty
\quad\text{as } p\to\infty .
\]

By assumption~\textbf{(b)}, the set of $y$ such that $\Pi^\star e(x^*,y)\neq 0$ has positive
$\mu$-measure. For such $y$, Lemma~\ref{ls1} implies that
\[
S_p(y):=\sum_{k=0}^{p-1}\|Df(x^*)^k \Pi^\star e(x^*,y)\|^2 \longrightarrow +\infty
\quad\text{as } p\to\infty .
\]
Since $(S_p(y))_{p\ge1}$ is nondecreasing in $p$, the monotone convergence theorem yields
$\int S_p(y)\,\mu(dy)\to +\infty$. Using~\eqref{eq:saddle-iter}, we conclude that
\[
\lim_{p\to\infty}
\int_{\mathbb{R}^p}
(Df(x^*)^p e(x^*,y))^\top \nabla^2V(x^*)
(Df(x^*)^p e(x^*,y))\,\mu(dy)
= -\infty .
\]
\medskip
\textsc{Step 3. Conclusion.}
Assume by contradiction that $\nu^0(\{x^*\})>0$.
Then the right-hand side of~\eqref{eq:saddle-exp} tends to $-\infty$ as $p\to\infty$,
while the left-hand side is nonnegative for all $p\ge1$.
This contradiction implies $\nu^0(\{x^*\})=0$.
\end{proof}

\subsection{Illustrative examples}

\subsubsection{A double-well potential}

Consider the one-dimensional potential
\[
V(x)=\frac14 x^4-\frac12 x^2,
\]
whose critical points are the two local minima $\pm 1$ and the unstable equilibrium at $0$.
Let $u_0>0$ and let $f$ be the time-$u_0$ map of the gradient flow $\dot x=-V'(x)$, i.e.
$f(x)=\Phi(x,u_0)$.
Then $f(0)=0$, $f(\pm1)=\pm1$, and $0$ is unstable since $f'(0)=e^{u_0}>1$.
For $u_0$ small enough, $V$ provides a Lyapunov function for the discretized dynamics and
tightness of invariant measures holds under the standing assumptions.
Theorem~\ref{ts6} then implies that any weak limit $\nu^0$ satisfies $\nu^0(\{0\})=0$, and hence
\[
\mathrm{supp}(\nu^0)\subset\{-1,1\}.
\]

\begin{figure}[ht]
    \centering
    \includegraphics[width=0.7\textwidth]{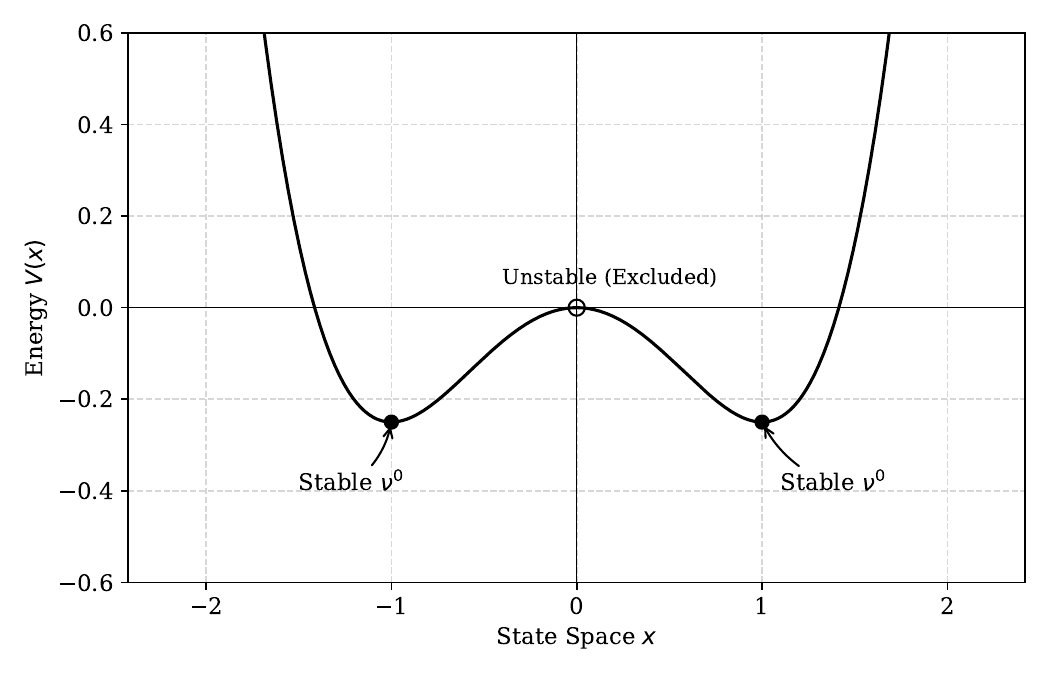}
    \caption{Potential landscape for the double-well system. Limiting invariant measures charge the stable equilibria $\{-1,1\}$ and exclude the unstable equilibrium at $0$.}
    \label{fig:potential}
\end{figure}

\subsubsection{A degenerate saddle: the lemniscate case}

The Lemniscate example provides a counterpoint showing that the non-degeneracy assumptions in
Theorem~\ref{ts6} cannot be dropped blindly.
In that system, the limiting measure concentrates at $(0,0)$, yet a direct computation yields
$\nabla^2V(0,0)=0$, and therefore
\[
\nabla^2 V(0,0) - Df(0,0)^{\top}\,\nabla^2 V(0,0)\,Df(0,0)=0.
\]
In this degenerate situation, the curvature-based mechanism captured by Theorem~\ref{ts6} is absent,
and a saddle point may indeed retain positive mass.
This illustrates that the theorem identifies a genuine threshold: exclusion holds when instability
and curvature interact through a non-degenerate second-order gap, and may fail when this gap
vanishes.

\section{Conclusion}

We have studied equilibrium selection for stochastic algorithms with constant step size in a
quasi-Feller framework, under minimal moment assumptions on the noise.
Our results show that the asymptotic behavior of invariant measures is governed by local Lyapunov
geometry and by the interaction between deterministic drift and persistent variance, rather than
by large-deviation effects or global smoothness assumptions.

At the level of localization, we prove that any weak limit of invariant measures is supported on
the Lyapunov plateau $\{V\circ f=V\}$ (and, under mild conditions, on the set of fixed points).
Going beyond localization, we establish sharp exclusion results for unstable equilibria.
In particular, strict local maxima and saddle points of the Lyapunov function are shown to be
excluded from the support of limiting invariant measures under explicit non-degeneracy conditions
linking curvature and noise in unstable directions.

These conclusions apply to a variety of discontinuous models, including projection-based
algorithms, vector quantization schemes such as Lloyd's algorithm, and noisy best-response
dynamics in coordination games.
They provide a unified probabilistic explanation for equilibrium selection in settings that fall
outside the scope of classical Feller theory or approaches based on small-noise large deviations.
From this perspective, selection in constant-step dynamics appears as a robust local phenomenon
driven by typical fluctuations rather than rare events.

Several directions remain for future work.
It would be natural to extend the analysis to broader classes of perturbations, to relax the
Lyapunov assumptions (for instance by allowing weaker descent or non-isolated plateaus), and to
investigate infinite-dimensional or constrained-state-space settings.
More broadly, we hope that the quasi-Feller viewpoint, together with the invariant-measure
approach to constant-step dynamics, will be useful for analyzing stability and selection in
stochastic systems operating under persistent noise.

\bibliography{PDS}

\appendix
\section{Additional technical material}\label{app:sec33}

This appendix gathers examples and derivations that were part of Section~\ref{lemn} in an extended
version, and are moved here for readability.

\subsection{A discontinuous contracting map}\label{app:discont-contract}

Let $f$ be a Borel function from $\mathbb{R}^d$ to $\mathbb{R}^d$ such that
\[
|f(x)|\leq a|x|,\quad 0<a<1.
\]
This condition is fulfilled by discontinuous Borel functions whose graph is of the form shown in
Figure~\ref{fig1}.

\begin{figure}[!ht]
\centering
\begin{tikzpicture}[scale=0.5, line cap=round, line join=round, >=Latex]

  % Parameters
  \def\a{0.5}   % cone slope
  \def\k{0.92}  % safety factor: keep |y| <= k*a*|x|
  \def\xmax{5.2}
  \def\ymax{3.7}

  % Axes
  \draw[->, very thin] (-\xmax,0) -- (\xmax,0);
  \draw[->, very thin] (0,-\ymax) -- (0,\ymax);

  % Cone y = ± a x
  \draw[thin] (-\xmax,{-\a*\xmax}) -- (\xmax,{\a*\xmax});
  \draw[thin] (-\xmax,{\a*\xmax}) -- (\xmax,{-\a*\xmax});

  \node[above right, font=\small] at (3.2,{\a*4.8}) {$y=ax$};
  \node[above right, font=\small] at (3.2,{-\a*6.8}) {$y=-ax$};

  % Discontinuous function f: pieces separated by gaps
  \draw[thin] plot[smooth] coordinates {
    (-4.9,{-\k*\a*4.9*0.98}) (-4.6,{-\k*\a*4.6*0.85})
    (-4.3,{-\k*\a*4.3*0.95}) (-4.0,{-\k*\a*4.0*0.80})
  };
  \draw[thin] plot[smooth] coordinates {
    (-3.9,{\k*\a*3.9*0.95}) (-3.6,{\k*\a*3.6*0.70})
    (-3.3,{\k*\a*3.3*0.92}) (-3.1,{\k*\a*3.1*0.78})
  };
  \draw[thin] plot[smooth] coordinates {
    (-3.0,{-\k*\a*3.0*0.90}) (-2.7,{-\k*\a*2.7*0.65})
    (-2.4,{-\k*\a*2.4*0.88}) (-2.2,{-\k*\a*2.2*0.72})
  };
  \draw[thin] plot[smooth] coordinates {
    (-2.1,{\k*\a*2.1*0.92}) (-1.9,{\k*\a*1.9*0.55})
    (-1.6,{\k*\a*1.6*0.88}) (-1.4,{\k*\a*1.4*0.70})
  };
  \draw[thin] plot[smooth] coordinates {
    (-1.3,{-\k*\a*1.3*0.90}) (-1.1,{-\k*\a*1.1*0.55})
    (-0.9,{-\k*\a*0.9*0.88}) (-0.7,{-\k*\a*0.7*0.72})
  };
  \draw[thin] plot[smooth] coordinates {
    (-0.6,{\k*\a*0.6*0.80}) (-0.45,{\k*\a*0.45*0.55})
    (-0.25,{\k*\a*0.25*0.85}) (0.0,0.0)
  };
  \draw[thin] plot[smooth] coordinates {
    (0.12,{\k*\a*0.12*0.90}) (0.30,{\k*\a*0.30*0.55})
    (0.55,{\k*\a*0.55*0.88}) (0.75,{\k*\a*0.75*0.70})
  };
  \draw[thin] plot[smooth] coordinates {
    (0.85,{-\k*\a*0.85*0.90}) (1.05,{-\k*\a*1.05*0.55})
    (1.25,{-\k*\a*1.25*0.88}) (1.45,{-\k*\a*1.45*0.70})
  };
  \draw[thin] plot[smooth] coordinates {
    (1.55,{\k*\a*1.55*0.92}) (1.85,{\k*\a*1.85*0.55})
    (2.05,{\k*\a*2.05*0.90}) (2.25,{\k*\a*2.25*0.72})
  };
  \draw[thin] plot[smooth] coordinates {
    (2.35,{-\k*\a*2.35*0.92}) (2.65,{-\k*\a*2.65*0.55})
    (2.95,{-\k*\a*2.95*0.90}) (3.25,{-\k*\a*3.25*0.72})
  };
  \draw[thin] plot[smooth] coordinates {
    (3.35,{\k*\a*3.35*0.92}) (3.70,{\k*\a*3.70*0.55})
    (4.20,{\k*\a*4.20*0.90}) (4.80,{\k*\a*4.80*0.60})
  };

  \node[right, font=\small] at (5,1.5) {$y=f(x)$};

\end{tikzpicture}
\caption{Sublinear Borel function.}
\label{fig1}
\end{figure}
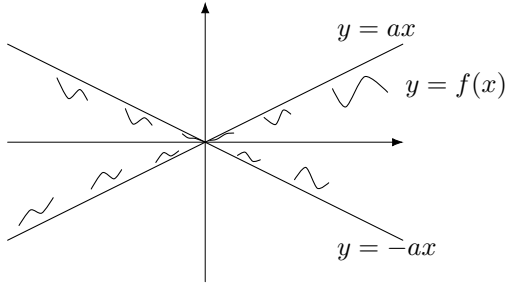

Under the assumptions detailed in Proposition~\ref{ps21}, and using the Lyapunov function
$V(x):=|x|^2$, it follows that any sequence of stationary distributions converges to $\delta_0$.
Indeed, when $g$ is bounded and continuous, the composition $g\circ f$ is $\delta_0$-a.s.\
continuous because $f$ is continuous at $0$ (and may be continuous only at that point).

\subsection{One-dimensional Lloyd's algorithm}\label{app:lloyd-1d}

\subsubsection*{Model and Lyapunov function}

This algorithm aims to determine an optimal $n$-level quantizer for a real random variable $X$
with diffuse distribution $\mu$ and finite second moment, assuming
${\rm Conv}({\rm supp}\,\mu)=[m,M]$.

Define
\[
\mathcal{O}_n:=\left\{(x^1,\cdots,x^n)\,:\, m<x^1<\cdots< x^n <M\right\}.
\]
For $x=(x^1,\cdots,x^n)\in \mathcal{O}_n$, define $f(x)=(v^1,\dots,v^n)$ with
\[
v_i:=\frac{1}{2}\bigl(g(x^{i-1},x^{i})+g(x^{i},x^{i+1})\bigr),
\qquad x^0:=m,\ x^{n+1}:=M,
\]
where, for $m\le \alpha\le \beta\le M$,
\[
g(\alpha,\beta):=\frac{\int_{\alpha}^{\beta} u\,\mu(du)}{\mu((\alpha,\beta])},
\quad \alpha<\beta,
\qquad g(\alpha,\alpha):=\alpha.
\]
Define the quantizer $Q_x$ by $Q_x(u):=g(x^{i-1},x^i)$ if $u\in[x^{i-1},x^i)$, and set
\[
V(x):=\mathbb{E}\bigl(Q_x(X)-X\bigr)^2
=\sum_{i=1}^n \int_{x^{i-1}}^{x^i}\bigl(g(x^{i-1},x^i)-u\bigr)^2\,\mu(du).
\]

\subsubsection*{Auxiliary lemmas}

\begin{lemma}\label{lzeb}
Fix $a^1,\cdots,a^n$ in $\mathbb{R}$.
The function
\[
(b^1,\cdots,b^{n+1})\mapsto \sum_{i=1}^{n+1}\int_{a^{i-1}}^{a^i}(b^i-u)^2\,\mu(du)
\]
is minimized at $b^i=g(a^{i-1},a^{i})$ for $1\le i\le n+1$, with $a^0:=m$ and $a^{n+1}:=M$.
\end{lemma}

\begin{lemma}\label{ll2}
Fix $b^1,\cdots,b^{n+1}$ in $\mathbb{R}$.
The function
\[
(a^1,\cdots,a^n)\mapsto \sum_{i=1}^{n+1}\int_{a^{i-1}}^{a^i}(b^i-u)^2\,\mu(du)
\]
is minimized at $a^i=\frac{b^i+b^{i+1}}{2}$ for $1\le i\le n$.
\end{lemma}

\subsubsection*{Lyapunov descent}

Combining Lemmas~\ref{lzeb} and \ref{ll2} yields
\[
V\circ f(x)\le V(x),
\qquad\text{and}\qquad
V\circ f(x)=V(x)\Rightarrow f(x)=x.
\]
Further remarks and uniqueness results under log-concavity assumptions can be found in the
references cited in the main text.

\subsection{Multi-dimensional Lloyd's algorithm}\label{app:lloyd-md}

We recall the definition of the Vorono\"{\i} mosaic and the standard Lyapunov descent argument.

\begin{definition}\label{mo}
Let $C$ be a nonempty convex and compact subset of $\mathbb{R}^d$.
The Vorono\"{\i} mosaic associated to $x:=(x^1,\cdots,x^{n})\in(\mathbb{R}^d)^n$ is the family
$\left(C_i(x)\right)_{1\leq i \leq n}$ of open convex subsets of $\mathbb{R}^d$ defined by:
\[
\left\{
\begin{array}{ll}
C_i(x):=\left\{u\in C:\ x^j\neq x^i \Rightarrow |x^i-u|<|x^j-u|\right\}
& \text{if } i=\min\{k:\ x^k=x^i\},\\
C_i(x):=\emptyset & \text{if there exists } j<i \text{ with } x^j=x^i.
\end{array}
\right.
\]
\end{definition}

The multidimensional analogues of Lemmas~\ref{lzeb}--\ref{ll2} are as follows.

\begin{lemma}\label{ll10}
Fix $a:=(a^1,\cdots,a^{n})$.
The function
\[
(b^1,\cdots,b^{n})\mapsto \sum_{i=1}^{n}\int_{C_i(a)}|b^i-u|^2\,\mu(du)
\]
is minimized at $b^i=\frac{\int_{C_i(a)}u\,\mu(du)}{\mu(C_i(a))}$.
\end{lemma}

\begin{lemma}\label{ll20}
Fix $(b^1,\cdots,b^n)$.
The function
\[
(a^1,\cdots, a^{n})\mapsto \sum_{i=1}^{n}\int_{C_i(a)}|b^i-u|^2\,\mu(du)
\]
is minimized at $a^i=b^i$.
\end{lemma}

Using Lemmas~\ref{ll10}--\ref{ll20}, one obtains $V\circ f\le V$ and $\{V\circ f=V\}=\Fix(f)$
as stated in the main text.

\subsection{The tent map example}\label{app:tent}

The ``tent'' map is defined by
\[
\left\{
\begin{array}{l}
f(x)=2x,\quad x\in \left[0,\frac{1}{2}\right],\\[2mm]
f(x)=2(1-x),\quad x\in \left(\frac{1}{2},1\right],\\[2mm]
f(x)=0,\quad x\in \mathbb{R}\setminus [0,1].
\end{array}
\right.
\]
It is a classical chaos generator and preserves the Lebesgue measure on $[0,1]$.
If $V$ is continuous and satisfies $V\circ f\le V$, then $V$ must be constant on $[0,1]$.
This example illustrates that in the absence of a nonconstant Lyapunov function, the localization
principle may become uninformative.

\end{document}